\documentclass{article}

\usepackage{arxiv}
\usepackage[T1]{fontenc}    
\usepackage{hyperref}       
\usepackage{url}            
\usepackage{booktabs}       
\usepackage{amsfonts}       
\usepackage{nicefrac}       
\usepackage{microtype}      
\usepackage{lipsum}
\usepackage{graphicx}
\usepackage{amsmath}
\usepackage{amsthm}
\usepackage{epstopdf}
\usepackage[caption=false]{subfig}
\usepackage{float}

\usepackage[longnamesfirst,sort]{natbib}
\bibpunct[, ]{(}{)}{;}{a}{,}{,}

\usepackage[ruled]{algorithm2e}
\usepackage{pgfplots}
\usepgfplotslibrary{groupplots}

\newtheorem{theorem}{Theorem}[section]

\newtheorem{proposition}[theorem]{Proposition}

\newtheorem{definition}[theorem]{Definition}

			  \def\X{\mathcal{X}}
				\def\H{\mathcal{H}}
				\def\S{\mathcal{S}_E}
				\def\N{\mathcal{N}}
			  
			  \def\B{\mathcal{B}}
			  \def\p1{P_1}
			  \def\p2{P_2}

\title{An exact method for optimizing two linear fractional functions over the efficient set of a Multiobjective Integer Linear Fractional Program }

\author{
 Yacine Chaiblaine \\
   USTHB, LaROMaD Laboratory \\ 
  Bp 32 El Alia, 16111, Algeria \\ 
  \texttt{ychaiblaine@usthb.dz} 
   \And 
   Mustapha  Moula\"{i} \\
  USTHB, LaROMaD Laboratory \\ 
  Bp 32 El Alia, 16111, Algeria \\
  \texttt{mmoulai@usthb.dz} \\
  \And
   Yasmine Cherfaoui \\
  USTHB, LaROMaD Laboratory\\ 
  Bp 32 El Alia, 16111, Algeria \\
  \texttt{ycherfaoui@usthb.dz} 
}
\begin{document}
\maketitle
\begin{abstract}
In this paper, an exact method is proposed to optimize two fractional linear functions over the efficient set of a fractional multiobjective linear problem $(MOILFP)$. This type of problems is encountered when there are two decision makers and each has his own utility function that he wants to optimize over the efficient set of multiobjective problem. The proposed method uses Branch and Bound method combined with a cutting plane technique to find the efficient solutions for both utility functions and $(MOILFP)$ without going through all the efficient solutions of the two problems. An illustrative example and a computational study are reported.
\end{abstract}


\section{Introduction}

Optimizing a nonlinear or linear function over the efficient set is an interesting area of multiobjective programming \citep{Miettinen99}. It is a simple way to avoid listing all efficient solutions by evaluating and distinguishing them from each other using a function that summarizes the preferences of decision makers.

First considered by Philip \citeyearpar{philip72}, the problem of optimization over the efficient set has since attracted the attention of several authors, among whom Jorge \citeyearpar{jorge09} who proposed an exact method that successively solves single-objective programs, Zerdani and Moula\"{i}  \citeyearpar{Zerdani11} have optimized a linear function over the integer efficient set of $MOILFP$ using Ehrgott's efficiency test \citeyearpar{Ehrgott97}. More recently, Drici et al. \citeyearpar{DOM18} also proposed an exact method that uses the well-known concept of branch-and-bound combined with efficient cuts technique and Liu and Ergoth  \citeyearpar{LE18} have presented primal and dual algorithms to solve this problem. Other entries can be found in Yamamoto's survey \citeyearpar{yamamoto02}.

Most of the time, in this type of problems, we have several decision makers and each one of them can have a utility function. Cherfaoui and Moula\"{\i} \citeyearpar{CM19} have treated the case where there are two linear functions to optimize over the efficient set of an integer linear multiobjective problem. However, in practice, measuring a quality, a profitability, a probability..., are described as  fractional functions. Fractional programming is encountered in several areas of application such as : stock cutting problem \cite{stockcutt}, shape-quality optimization \cite{shapequality}, clustering problems \cite{cluster}, etc., more applications of fractional programming can be found in \cite{sheibelsurvy}, Stancu-Minasian's bibliography \citeyearpar{stancu19}.

In this article, we generalize the latter method \citep{CM19} where we have two linear fractional functions $(BOILFP)$ that we want to optimize over the efficient set of $(MOILFP)$. This problem was encountered in a real world problem summarized as follow: A company "A" wants to subcontract a part of its production. In the specifications, company A has two criteria: optimizing the profitability and the quality, each criterion is represented by a linear fractional function. On another hand, a Company "B" wants to take over the project and at the same time has several fractional linear objective functions. The company "B" wants to know if there is one or more solutions to benefit both companies. In other words, finding the solutions that are efficient for both the $(BIOLFP)$ problem and the $(MOILFP)$ problem. One way to solve this problem is to enumerate the  efficient sets of both $(BIOLFP)$ and $(MOILFP)$ and find the intersection between the two efficient sets. This method has an important computational cost as we will show it in the computational study.

The proposed method is an exact branch and cut procedure that provides the exact  efficient solutions for both  $(BIOLFP)$ and $(MOILFP)$. This procedure has the advantage to avoid useless exploration of domain by using cuts that eliminates dominated solutions, pruning by the same way useless nodes in the tree. It also uses efficiency tests that allow only the good solutions to be selected.
\section{Definitions and preliminaries}

Linear Fractional Programs $(LFP)$ are very important because of their contribution to real problems \citep{stancu12,baja13}, their general formulation is :

\begin{equation}
	\label{lfp}
	(LFP)\left\{\begin{array}{lll}
	\max f_1(x)=\dfrac{p^1x+\alpha^1}{q^1x+\beta^1} \\\
	s.t.\\
	x\in \mathcal{X}
	\end{array}\right.
\end{equation}

where  $\mathcal{X} = \{x \in \mathbb{R}^n|Ax \leq b,\ x \geq 0\}$, $p^1$, $q^1$ are $n-$vectors, $\alpha^1$ and $\beta^1$ are scalars, $b$ is an $m-$vector, $A$ is an $n\times m-$matrix and $x$ is the  $n-$vector decision.

The set $\mathcal{X}$ is assumed to be  nonempty convex polytope and the set $\mathcal{D}=\mathcal{X}\cap\mathbb{Z}^n$ is assumed not empty as well.

Several authors proposed approaches that can solve $(LFP)$ efficiently. In the first place, Charnes and Cooper \citeyearpar{CC62} proposed a  variable transformation that turns the problem into a linear program.
Cambini and Martein \citeyearpar{ CM86} proposed a simplex procedure improved from the procedure of Martos \citeyearpar{martos64}. This procedure uses the reduced gradient to move towards the optimum for a base $\mathcal{B}_l$ and associated solution $x^{*(l)}$ using the following notations:

\[\begin{array}{lll}
\nu^{1(l)}= p^{1} - p^{1}_{\mathcal{B}_l}\mathcal{B}_{l}^{-1} A \\
\mu^{1(l)}=q^{1} - q^{1}_{\mathcal{B}_l}\mathcal{B}_{l}^{-1} A, \\
P^1(x)=p^1x^{*(l)}+\alpha^1\\
Q^1(x)=q^1x^{*(l)}+\beta^1\\
\end{array}
\]

the reduced gradient for the problem (\ref{lfp})  can be formulated as:

\[\boldsymbol{\gamma^{1(l)}}= Q^1(x)\nu^{1(l)}- P^1(x)\mu^{1(l)}\]

The reduced gradient gives the direction of growth of the objective function $f_1(x)$, so if for all index $j$  from the non-basic index set $\mathcal{N}_l$, $\boldsymbol{\gamma^{1(l)}_j} \leq 0$ then the function $f_1(x)$ cannot increase.

Most of real life problems can be modelled as multiobjective problems with integer numbers, therefore let us consider the following multiobjective program:

\begin{equation}
	\label{moilfp}
	(MOILFP)\left\{\begin{array}{lll}
	\max Z_1(x)=\cfrac{c^1x+c^1_0}{d^1x+d^1_0}\\
	\noalign{\smallskip}
	\max Z_2(x)=\cfrac{c^2x+c^2_0}{d^2x+d^2_0}\\\noalign{\smallskip}
	\vdots \\
	\max Z_k(x)=\cfrac{c^kx+c^k_0}{d^kx+d^k_0}\\\noalign{\smallskip}
	s.t.\\\noalign{\smallskip}
	x\in \mathcal{D}
	\end{array}\right.
\end{equation}

where $k \geq 2$; $c^i, d^i$ are $n-$vectors; $c^i_0, d^i_0$ are scalars for each $i \in\{1, 2, \ldots , k\}$. Throughout this article, we assume that
$ d^ix+d^i_0 > 0$ over $\mathcal{X}$ for all $i\in\{1, 2, \ldots , k\}$.

\begin{definition}
	A solution  $x \in \mathcal{D}$ is called an \textbf{efficient} solution for $(MOILFP)$, if there exists no point $y \in \mathcal{D}$ such that $Z_i(y) \geq Z_i(x)$, for all $i\in\{1,\ldots, k\}$ and $Z_i(y) > Z_i(x)$ for at least one $i \in \{1,\ldots, k\}$.
\end{definition}

Some researchers tackled this problem in their papers especially in the continuous case; Kornbluth  and Steuer \citeyearpar{KS81} presented a simplex-based solution procedure to find all weakly efficient vertices, while Costa \citeyearpar{costa07} proposed a new technique to optimize a weighted sum of the linear fractional objective functions. Finally, Cambini et al. \citeyearpar{CMS99} wrote a survey on the biobjective fractional problems.

In general, since it is computationally heavy, authors avoid trying to generate the whole efficient set. However, Chergui and Moula\"{\i} \citeyearpar{CM08} used a branch-and-cut procedure to generate the whole efficient set. In their paper, they combined branch-and-bound exploration with a cutting plane technique to generate the efficient set without enumerating the whole integer domain.

In this paper, we propose a method that generates a part of the efficient set of $(MOILFP)$ that is defined by decision makers preferences. Therefore, let us assume the existence of two utility functions $f_1$ and $f_2$:

\begin{equation}
	\label{bio}
	(BOILFP)
	\left\{
	\begin{array}{ll}
	\max f_1(x)=\cfrac{p^1x+\alpha^1}{q^1x+\beta^1}\\\noalign{\smallskip}
		\max f_2(x)=\cfrac{p^2x+\alpha^2}{q^2x+\beta^2}\\\noalign{\smallskip}
	s.t.\\
	x\in \mathcal{D}

	\end{array}
	\right.
\end{equation}

where  $p^1,~ p^2,~ q^1,~q^2$ are $n-$vectors; $\alpha^1,~ \beta^1$ and $\alpha^2,~ \beta^2$ are scalars. Throughout this article, we also assume that $ q^1x+\beta^1 > 0$  and $q^2x+\beta^2>0$ over $\mathcal{X}$ and the following notations are used for  $\mathcal{B}_l,~\mathcal{N}_l$ and $x^{*(l)}$ :

\[
\begin{array}{lll}
P^s(x)=p^sx+\alpha^s,&~Q^s(x)=q^sx+\beta^s,&~s=1,2,\\
z^1_i(x)=c^ix+c^i_0,&~z^2_i(x)=d^ix+d^i_0,&~i\in\{1,2,\ldots,k\},\\
\nu^{s(l)}= p^{s} - p^{s}_{\mathcal{B}_l}\mathcal{B}_l^{-1} A,&~\mu^{s(l)}=q^{s} - q^{s}_{\mathcal{B}_l}\mathcal{B}_l^{-1} A,&~s=1,2, \\
\eta^{i(l)}= c^{i} - c^{i}_{\mathcal{B}_l}\mathcal{B}_l^{-1} A,&~\vartheta^{i(l)}=d^{i} - d^{i}_{\mathcal{B}_l}\mathcal{B}_l^{-1} A,&~i\in\{1,2,\ldots,k\}. \\
\end{array}\]

Therefore, the problem that we propose to solve is:

\begin{equation}
	\label{bmlf}
	(B_E)
	\left\{
	\begin{array}{ll}
	\max f_1(x)=\cfrac{p^1x+\alpha^1}{q^1x+\beta^1}   \\ \noalign{\smallskip}
		\max f_2(x)=\cfrac{p^2x+\alpha^2}{q^2x+\beta^2} \\ \noalign{\smallskip}
	s.t.\\
	x\in \mathcal{X}_E, 
	\end{array}
	\right.
\end{equation}

where $\mathcal{X}_E$ is the set of efficient solutions of $(MOILFP)$,  we also note $\mathcal{X}_{E'}$ the set of efficient solutions of the $(BOILFP)$.
The problem $(B_E)$ has a real practical utility, yet it has not been considered. It arises, for example, whenever two firms have to optimize their respective interest and a common problem.

\subsection{Efficiency test for multi-objective linear fractional programs}

The efficiency test of integer feasible solutions $ x^*$ of the problem (\ref{moilfp}) is verified by solving the mixed integer linear program (\ref{MM}).

\begin{equation}\label{MM}
MM( x^*) ~ ~ ~  \left\{\begin{array}{llll}
\max ~ ~ \Psi = \sum\limits_{i=1}^{k} \psi_i\\
s.t.\\
~ ~ ~ \; ~ ~ [c^i-Z_i( x^*)d^i]x - \psi_i = Z_i( x^*)d_0^i - c_0^i, \; i=1,...,k \\
~ ~ ~ ~ ~ ~ ~ x\in \mathcal{D}; \\
~ ~ ~ ~ ~ ~ ~ \psi_{i} \geq  0,  \text{ for all } i\in \{1,...,k\}
\end{array}
\right.
\end{equation}

\noindent where $Z_i( x^*)=\dfrac{c^i x^*+c_0^i}{
d^i x^*+d_0^i}$ and $\psi= (\psi_i)_{i\in\{1,...,K\}}\in {\mathbb R}$.

\begin{theorem}
A feasible  solution $x^*$ of  problem (\ref{moilfp}) is efficient if and only if the optimal objective value of problem $MM(x^*)$ is zero.
\end{theorem}

\begin{proof}

Let $(y,\psi)$ be any feasible solution of problem $MM( x^*)$, then for all $ i=1,\ldots,k $, we have:
\begin{equation*}
\begin{array}{l}
    \psi_i\geq 0,  \\ \Longleftrightarrow
    (c^i-Z_i( x^*)d^i)y - Z_i( x^*) d_0^i + \lambda c_0^i\geq 0, \\
  \Longleftrightarrow Z_i(y)\geq Z_i( x^*).
\end{array}
\end{equation*}
assume that $\Psi=\sum\limits_{i=1}^{k} \psi_i\neq 0$, then there exists $r\in \{1,\ldots,k\}$ such that
\begin{equation*}
\begin{array}{l}
   \psi_r > 0 \Longrightarrow (c^r-Z_r( x^*)d^r)y - Z_r( x^*) d_0^r +  c_0^r > 0\\
   \qquad \quad \Longrightarrow Z_r(y)> Z_r( x^*),
\end{array}
\end{equation*}
therefore, there exists a feasible solution $y\in \mathcal{D}$ such that $Z_{i}(y)\geq Z_{i}( x^*)$ for all $i\in \left\{ 1,...,k\right\}$ and $Z_{r}(y)>Z_{r}( x^*)$ for at least one $r\in \left\{ 1,...,k\right\}$. Hence, $ x^*$ is not efficient.

Conversely, let ($y,\psi$) be an optimal  solution of $MM( x^*)$, where $\Psi=0$, then we have

$\sum\limits_{i=1}^{k} \psi_i=0\Longrightarrow  \psi_i=0,\quad i=1,\ldots,k $,   it follows that for all criteria $(Z_{i})_{i=1,\ldots,k}$
\begin{equation*}
\begin{array}{l}
    (c^i-Z_i( x^*)d^i)y - Z_i( x^*) d_0^i +  c_0^i=0, \\
\Longrightarrow  c^iy +  c_0^i= Z_i( x^*)(d^iy +  d_0^i),\\
\Longrightarrow   Z_i(y)=Z_i( x^*).
\end{array}
\end{equation*}
Hence, the criterion vector $(Z_1( x^*),Z_2( x^*),\ldots,Z_k( x^*))$ is not dominated  , then $ x^*$ is efficient.
\end{proof}

\section{Methodology, Algorithm and Theoretical Results}
The naive approach to solve $(B_E)$ consists in enumerating the elements of $\mathcal{X}_E$ and $\mathcal{X}_{E'}$ and determine their intersection.
In what follows, we propose a method that avoids the enumeration of both and gives the exact solution set of $(B_E)$. This method is a branch-and-cut procedure that combines the branch-and-bound exploration technique and cutting plane method. The solution set of $(B_E)$ is noted $(\S)$.

In this paper, we propose to use an equivalent cutting plane method of Chergui and Moula\"{\i} \citeyearpar{CM08} to remove  solutions that do not belong to $\mathcal{X}_E$ or $\mathcal{X}_{E'}$.  Our method can be summarized as follows:

\begin{description}
\item[\textbf{The branching process}] The search for efficient solutions is an implicit exploration structured as a tree where $f_1$ or $f_2$ can be chosen to be optimized at each node. The choice of function $f_1$ or $f_2$ will not change the result set $\S$, however, it will alter the search tree and change the evolution of $ \S $ but not the final result.

At each node $l$ of the tree we optimize the function $f_1$ over the subdomain corresponding to the node $\mathcal{X}_l$.

\[\label{prob2}
(LFP_l)\left\{\begin{array}{ll}
\max f_1(x)=\cfrac{p^1x+\alpha^1}{q^1x+\beta^1}\\\noalign{\smallskip}
s.t.\\
x\in\mathcal{X}_l
\end{array}\right.
\]

where $\mathcal{X}_0=\mathcal{X}$ and $\mathcal{X}_{l}$ is subset of the original set $\mathcal{X}$ to explore at node $l$ of the tree.

The functions $f_1$ and $f_2$ are linear fractional. To solve the $(LFP_l)$, we can either use the method presented in \citep{CC62} or another simplex method like those presented in \citep{martos64,CM86} .

After solving $(LFP_l)$, three cases can occur:

\begin{description}
	\item[\textbf{$(LFP_l)$ has no solution}]  As for  a classical Branch-and Bound process the node has no descendant.

	\item[\textbf{$(LFP_l)$ has a non-integer solution $x^{*(l)}$ }] If  the solution of $(LFP_l)$ is not integer, the approach is the same as the Branch-and-Bound for integer programming, which means that the node has two descendants $l_1$ and $l_2$ such as:
	Let $r$ be an index where $ x_r^{*(l)} $ is not integer then for $l_1$ we put $\mathcal{X}_{l_{1}}=\left\{x\in\mathcal{X}_{l} | x_r \leq \left\lfloor {x_r^{*(l)}}\right\rfloor   \right\}$ and for $l_{2}$ we put $\mathcal{X}_{l_2}=\left\{x\in\mathcal{X}_l | x_r \geq \left\lceil {x_r^{*(l)}}\right\rceil \right\}$.

	\item[\textbf{$(LFP_l)$ has an integer solution $x^{*(l)}$}] If after calculating the solution $x^{*(l)}$, this solution is integer then two steps are performed:
	
	\begin{itemize}
		\item \textbf{\emph{The first step is to test the membership of $x^{*(l)}$ to $\S$:}}			
			
		Since there are two  efficiency tests to take into account:
			\[\begin{array}{ll}\label{prob21}(T^1_{x^{*(l)} })
\left\{\begin{array}{ll}
\max \sum_{i=1}^{i=k} w_i\\\noalign{\smallskip}
s.t.\\
\left(c^i-Z_i(x^{*(l)})d^i\right)x-w_i=Z_i(x^{*(l)})d^i_0-c^i_0,~i=\overline{1,k}\\
x\in\mathcal{D}
\end{array}\right.
\\
(T^2_{x^{*(l)} })
\left\{\begin{array}{ll}
\max  v_1+v_2\\\noalign{\smallskip}
s.t.\\
\left(p^1-f_1(x^{*(l)})q^1\right)x-v_1=f_1(x^{*(l)})\beta^1-\alpha^1~~~~~~~~~~~~~\\
\left(p^2-f_2(x^{*(l)})q^2\right)x-v_2=f_2(x^{*(l)})\beta^2-\alpha^2\\
x\in\mathcal{D}
\end{array}\right.
\end{array}
\]
			
The first program is for testing the efficiency in the program	$(MOILFP)$ and the second for the program $(BIOLFP)$. If both programs  have a zero as an optimal value it means that the solution  $	x^{*(l)}$ is efficient for 	$(MOILFP)$ and $(BIOLFP)$ and thus $ x^{*(l)}\in\S$.

		\item \emph{\textbf{The second is to search the descendants of the node l:}} The node in this case has either no descendant or one direct descendant depending on the growth direction of the functions $Z_i, i\in\{1,\ldots,k\}$ and $f_2$. If there is no increasing for all the functions $Z_i$ in the domain $\mathcal{X}_l$ then $x^{*(l)}$ dominates all the domain $\mathcal{X}_l$ in terms of $(MOILFP)$ and the node is fathomed. Similarly if there is no solution that is at least better than $x^{*(l)}$ in the criterion $f_2$, then the solution $x^{*(l)}$ dominates the whole domain $\mathcal{X}_l$ in terms of $(BOILFP)$ and the node is fathomed.

On the other hand, if there is a possible improvement of one of the functions $Z_i, i\in\{1,\ldots,k\}$ and $f_2$, an efficient solution in terms of $(MOILFP)$ and $(BOILFP)$ might still exists in $ \mathcal{X}_l $ and the node has  descendants.	
So, the cutting plane method that we use permit to remove solutions that are dominated by $x^{*(l)}$ criteria vectors.

Let $\mathcal{B}_l$  and $\mathcal{N}_l$  be the sets of basic variables and non-basic variables indexes of $x^{*(l)}$ respectively. In what follows we will use these notations for $j\in \mathcal{N}_l$:

\[
\begin{array}{lll}
\boldsymbol{\gamma^{2(l)}}= Q^2(x^{*(l)})\nu^{2(l)}- P^2(x^{*(l)})\mu^{2(l)}\\
\boldsymbol{\lambda^{i(l)}}= z^2(x^{*(l)})\eta^{i(l)}-z^1(x^{*(l)})\vartheta^{i(l)},~i\in\{1,2,\ldots,k\}
\end{array}\]

\noindent Since $\boldsymbol{\gamma^{1(l)}}$ and $\boldsymbol{\gamma^{2(l)}}$ represent the reduced gradient vectors for $f_1$ and $f_2$, and $\boldsymbol{\lambda^{i(l)}}$ corresponds to the reduced gradient vectors for the criterion $Z_i$  for  each $i\in\{1,2,\ldots,k\}$. We define the sets $\mathcal{H}_l$ and $\mathcal{H}'_l$  for each node $l$ as follows:

\begin{equation}\label{coup1}
\mathcal{H}_l=\left\{j \in \mathcal{N}_l|\exists i\in\{1,...k\}; \boldsymbol{\lambda^{i(l)}_j}>0\right\}\cup\left\{j \in \mathcal{N}_l|  \boldsymbol{\lambda^{i(l)}_j}=0~,\forall i\in\{1,\ldots,k\} \right\}
\end{equation}

\begin{equation}\label{coup2}
\mathcal{H}'_l=\left\{j \in \mathcal{N}_l| \boldsymbol{\gamma^{2(l)}_j} >0\right\}\cup\left\{j \in \mathcal{N}_l|  \boldsymbol{\gamma^{2(l)}_j}=0 \text{ and } \boldsymbol{\gamma^{1(l)}_j}=0~ \right\}
\end{equation}

In the following theorem \ref{theo} we will prove that the cuts:
\[\sum_{j\in\mathcal{H}_l}x_j\geq1 \text{ and } \sum_{j\in\mathcal{H}'_l}x_j\geq1\]
are valid and eliminate solutions that are dominated by $x^{*(l)}$.

	\end{itemize}	
\end{description}

\noindent After finding an integer solution, we update $\S$ and calculate $\mathcal{H}_l,\mathcal{H}'_l$.  If $\mathcal{H}_l=\emptyset $ or $\mathcal{H}'_l=\emptyset $ it means that the remaining domain contains no efficient solution either in terms of $(MOILFP)$ or $(BOILFP)$, and the node $l$ is fathomed (See proposition \ref{prop}).

\noindent Otherwise, if  $\mathcal{H}_l\neq \emptyset $ and $\mathcal{H}'_l\neq \emptyset $ the node $l$ has one successor $l_0$ with $\mathcal{X}_{l_0}= \left\{x\in\mathcal{X}_l | \sum_{j\in\mathcal{H}_l}x_j\geq 1,~\sum_{j\in\mathcal{H}'_l}x_j\geq 1 \right\}$.
\end{description}

\noindent Thus, there exists two rules of fathoming nodes, the first rule is when the corresponding program $(LFP_l)$ is not feasible and the second is when  $\mathcal{H}_l=\emptyset$ or $\mathcal{H}'_l=\emptyset$ .


In the following algorithm, the nodes in the tree structure  \ref{algo} can be treated according to the backtracking or the depth first principles or any other exploration strategy.

\begin{algorithm}[H]
 \KwResult{The set $\S$ containing all the  efficient solutions of the program (\ref{bmlf}) }
 initialization\ $l = 0$ ,$\mathcal{X}_l=\left\{x\in\mathbb{R}^n| Ax\leq b \text{ and }x\geq 0\right\}$ and $\S=\emptyset$; \\
\While{there is a non-fathomed node $l$}
{
solve $(LFP_l)$ ;\\
\[(LFP_l)\left\{\begin{array}{ll}
\max & f_1(x) \\
s.t. \\ & x\in\mathcal{X}_l
\end{array}\right.
\]
\eIf{$(LFP_l)$ has an optimal solution $x^{*(l)}$}
{
\eIf{$x^{*(l)}$ is integer}
{Update  $\S$;\\
Construct the sets $\mathcal{H}_l$,~$\mathcal{H'}_l$  ;\\
\eIf{$\mathcal{H}_l=\emptyset$ or $\mathcal{H'}_l=\emptyset$ } {Fathom the node $l$}{ Add the cuts (\ref{coup1}) and (\ref{coup2})  to $l_0$ the successor of $l$;}
}
{
Choose an index $r$ such as  ${x_r^{*(l)}}$ is fractional. Then, split the program $(LFP_l)$ into two sub-programs, by
adding respectively the constraints $x_r \leq \left\lfloor {x_r^{*(l)}}\right\rfloor$ and $x_r \geq \left\lceil {x_r^{*(l)}}\right\rceil$ to  obtain $(LFP_{l_1} )$ and  $(LFP_{l_2} )$ ( $l_1 \geq l + 1$, $l_2 > l + 1$ and $l_1 \neq l_2$);
}
}
{Fathom the node l;}}
\caption{Biobjective optimization over MultiObjective Integer Linear Fractional Program} \label{algo}
\end{algorithm}

\section{Theoretical results}\label{sectiont}

The following theoretical tools show that the algorithm yields the solution set  $\S$ of the program (\ref{bmlf}) $(B_E)$ in a finite number of iterations.

\begin{theorem}\label{theo}
Assume that $\H_l \neq \emptyset$ and ${\mathcal{H}'}_l \neq \emptyset$ at the current integer solution $x^{*(l)}$. If $x \in\S,~ x \neq x^{*(l)}$ and   $x\in\mathcal{X}_l$, then $x \in \mathcal{X}_{l_1}$  ($l_1$ successor of $l$).
\end{theorem}

\proof
    Let $x \neq x^{*(l)}$ be an integer solution in domain $\mathcal{X}_l$ such as $x \notin \mathcal{X}_{l_1}$, two
    cases can occur:

\begin{itemize}
\item $x\notin\mathcal{X}_{l_1}$ because $x\in\left\{x\in\mathcal{X}_l|~\sum_{j\in\mathcal{N}_l/\mathcal{H}_l} x_j \geq 1\right\}$, this implies
  \begin{equation*}
  x\in \{x\in \mathcal{X}_l| \sum\limits_{j\in  \H_l}x_j< 1\}.
 \end{equation*}

\noindent Therefore, the following inequalities hold:
  \begin{equation*}
    \begin{array}{l}
       \quad\sum\limits_{j\in \H_l}x_j < 1 \\
       \sum\limits_{j\in \N_{l}\setminus \H_l}x_j \geq 1.
     \end{array}
  \end{equation*}

It follows that $x_{j}=0$ for all $j\in \H_{l}$ and $x_{j}\ge 1$ for at least one index $j\in \N_{l}\setminus \H_l$.

From the optimal simplex table corresponding to solution $x^{*(l)}$, the updated value of each objective function is written according to the non-basic indexes $j\in \N_l$ as follows:
\begin{equation*}
\begin{array}{cc}
z_i^1(x)=z^1_i(x^{*(l)}) + \boldsymbol{\delta_{j}}\eta^{i(l)}_j,\\
z_i^2(x)=z^2_i(x^{*(l)}) + \boldsymbol{\delta_{j}}\vartheta^{i(l)}_j,\\
\end{array}
\end{equation*}
\noindent where  $\boldsymbol{\delta_j}=\dfrac{x^{*(l)}_{\B_{l_{(r^*)}}}}{A^{r^*}_{j}}=\min \left\{ \dfrac{x^{*(l)}_{\B_{l_{(r)}}}}{A^r_{j}}|~ A^r_{j}>0\right\}$.

Thus, we have for all  $i~\in\{1,2,\ldots,k\}$ :

\begin{equation*}
\begin{array}{llll}
 Z_{i}(x)&= \dfrac{z^{1}_i(x)}{z^2_{i}(x)}=
  \dfrac{z^{1}_i(x^{*(l)}) + \boldsymbol{\delta_{j}}\eta^{i(l)}_j}{z^2_{i}(x^{*(l)}) + \boldsymbol{\delta_{j}}\vartheta^{i(l)}_j},
 \\
\end{array}
\end{equation*}

then we can write

 \begin{equation*}
 \begin{array}{l}
 Z_{i}(x)-Z_{i}(x^{*(l)})= \dfrac{z^{1}_i(x^{*(l)}) + \boldsymbol{\delta_{j}}\eta^{i(l)}_j}{z^2_{i}(x^{*(l)}) + \boldsymbol{\delta_{j}}\vartheta^{i(l)}_j}-\dfrac{z^{1}_i(x^{*(l)})}{z^2_{i}(x^{*(l)})}\\\noalign{\smallskip}

   \qquad= \boldsymbol{\delta_{j}}\dfrac{\left[z^2_{i}(x^{*(l)})( \eta^{i(l)}_j)- z^{1}_i(x^{*(l)})(\vartheta^{i(l)}_j)\right]}{ {z}_i^2(x^{*(l)})\left[z^2_{i} (x^{*(l)})+ \boldsymbol{\delta_{j}} (\vartheta^{i(l)}_j)\right]}\\\noalign{\smallskip}

   \qquad= \boldsymbol{\delta_{j}}\dfrac{\left[z^2_{i}(x^{*(l)})(\eta^{i(l)}_j)- z^{1}_i(x^{*(l)})(\vartheta^{i(l)}_j)\right]}{ {z}_i^2(x^{*(l)})z^2_{i} (x)}\\\noalign{\smallskip}
 \end{array}
 \end{equation*}

Since we have already the following notation:
\[  \boldsymbol{\lambda^{i(l)}}= z^2(x^{*(l)})\eta^{i(l)}-z^1(x^{*(l)})\vartheta^{i(l)},~i=1,\ldots,k  \]

As the components $\boldsymbol{\lambda^{i(l)}_j}\leq 0,$ for every index $j\in \N_{l}\backslash \H_{l}$

\begin{equation*}
\begin{array}{ll}
Z_{i}(x)-Z_{i}(x^{*(l)}) \leq 0,~ &i=1,\ldots,k
\end{array}
\end{equation*}

and  for at least one criterion $ i_0\in\{1,\ldots,r\},  Z_{i_0}(x)-Z_{i_0}(x^{*(l)}) < 0$.

Hence, the criterion vector $\left( Z_{1}(x),Z_{2}(x), \ldots,Z_{k}(x)\right)$ is dominated by the criterion vector
$\left(  Z_1(x^{*(l)}), Z_{2}(x), \ldots, Z_k  (x^{*(l)})  \right)$ then $x \notin \mathcal{X}_E $. Hence $ x \notin \S $

\item  $x\notin\mathcal{X}_{l_1}$  because $x\in\left\{x\in\mathcal{X}_l|~\sum_{j\in\mathcal{N}_l/\mathcal{H'}_l} x_j \geq 1\right\}$,
by the same way it implies that
  \begin{equation*}
  x\in \{x\in \mathcal{X}_l| \sum\limits_{j\in  \H'_l}x_j< 1\}.
 \end{equation*}

\noindent Therefore, the following inequalities hold:
  \begin{equation*}
    \begin{array}{l}
       \quad\sum\limits_{j\in \H'_l}x_j < 1 \\
       \sum\limits_{j\in \N_{l}\setminus \H'_l}x_j \geq 1.
     \end{array}
  \end{equation*}

It follows that $x_{j}=0$ for all $j\in \H'_{l}$ and $x_{j}\ge 1$ for at least one index $j\in \N_{l}\setminus \H'_l$.

From the optimal simplex table  corresponding to solution $x^{*(l)}$, the updated value of each objective function  is written according the non-basic indexes $j\in \N_l$ as follows:
\begin{equation*}
\begin{array}{cc}
P^2(x)=P^2(x^{*(l)}) + \boldsymbol{\delta_{j}}\nu^{2(l)}_j,\\
Q^2(x)=Q^2(x^{*(l)}) + \boldsymbol{\delta_{j}}\mu^{2(l)}_j,\\
\end{array}
\end{equation*}
\noindent where  $\boldsymbol{\delta_j}=\dfrac{x^{*(l)}_{\B_{l_{(r^*)}}}}{A^{r^*}_{j}}=\min \left\{ \dfrac{x^{*(l)}_{\B_{l_{(r)}}}}{A^r_{j}}|~ A^r_{j}>0\right\}$.

Thus, we have  :

\begin{equation*}
\begin{array}{llll}
 f_{2}(x)&= \dfrac{P^{2}_i(x)}{Q^2(x)}=
  \dfrac{P^{2}(x^{*(l)}) + \boldsymbol{\delta_{j}}\nu^{2(l)}_j}{Q^2_{i}(x^{*(l)}) + \boldsymbol{\delta_{j}}\mu^{2(l)}_j},
 \\
\end{array}
\end{equation*}

then we can write

 \begin{equation*}
 \begin{array}{l}
 f_{2}(x)-f_{2}(x^{*(l)})= \dfrac{P^{2}(x^{*(l)}) + \boldsymbol{\delta_{j}}\nu^{2(l)}_j}{Q^2(x^{*(l)}) + \boldsymbol{\delta_{j}}\mu^{2(l)}_j}-\dfrac{P^{2}_i(x^{*(l)})}{Q^2(x^{*(l)})}\\\noalign{\smallskip}

   \qquad= \boldsymbol{\delta_{j}}\dfrac{\left[Q^2(x^{*(l)})( \nu^{2(l)}_j)- P^{2}(x^{*(l)})(\mu^{2(l)}_j)\right]}{ {Q}^2(x^{*(l)})\left[Q^2_{i} (x^{*(l)})+ \boldsymbol{\delta_{j}} (\mu^{2(l)}_j)\right]}\\\noalign{\smallskip}

   \qquad= \boldsymbol{\delta_{j}}\dfrac{\left[Q^2_{i}(x^{*(l)})(\nu^{2(l)}_j)- P^{2}(x^{*(l)})(\mu^{2(l)}_j)\right]}{ {Q}^2(x^{*(l)})Q^2 (x)}\\\noalign{\smallskip}
 \end{array}
 \end{equation*}

Since we have already the following notation:
\[\boldsymbol{\gamma^{2(l)}}= Q^2(x^{*(l)})\nu^{2(l)}- P^2(x^{*(l)})\mu^{2(l)}  \]

As the components $\boldsymbol{\gamma^{2(l)}}_j\leq 0,$ for every index $j\in \N_{l}\backslash \H'_{l}$

\begin{equation*}
\begin{array}{ll}
f_2(x)-f_2(x^{*(l)}) \leq 0
\end{array}
\end{equation*}

\noindent Thus, $f_{2}(x)\leq f_{2}(x^{*(l)})$ , with  $f_{1}(x)\leq f_{1}(x^{*(l)})$ since $x^{*(l)}$ is the maximum of the  program $(LFP_l)$, then $x\notin \mathcal{X}_{E'}$. Hence $x\notin \mathcal{S}_{E}$.
\end{itemize}

\noindent Since $x\notin \mathcal{S}_E$ in both cases, then $x$ is not a solution of the program $(B_E)$.
\endproof

\begin{proposition}\label{prop}
Suppose that $\mathcal{H}_l=\emptyset$ or $\mathcal{H'}_l=\emptyset$ at the current integer solution $x^{*(l)}$  then there is no solution in the remaining domain that is not dominated by $x^{*(l)}$.
\end{proposition}

\proof
\noindent
\begin{itemize}
  \item Assume $\mathcal{H}_l = \emptyset$, then $\forall i \in \{1,\ldots, k\}$, $\forall j \in \mathcal{N}_l$, we have $\overline{\gamma}^i_j \leq  0$ and $\exists i_0 \in\{1,\ldots,r\}$ such that $\overline{\gamma}^{i_0}_j<0, \forall j \in \mathcal{N}_l$. So, $x^{*(l)}$ dominates all points $x$, $x \neq x^{*(l)}$ of domain $\mathcal{D}_l$.
  \item Now, assume that $\mathcal{H}'_l = \emptyset$, then $ \forall j \in \mathcal{N}_l,\  \overline{\lambda}^2_j <  0  \text{ or }  \overline{\lambda}^2_j=0 \text{ and } \overline{\lambda}^1_j < 0 $, adding to that $\overline{\lambda}^1_j <  0,~\forall j \in \mathcal{N}_l$ since it is an optimal solution for $(LP_l)$, $x^{*(l)}$ becomes the most preferred solution  in the domain $\mathcal{D}_l$.
\end{itemize}
\endproof

\begin{theorem}
The algorithm terminates in a finite number of iterations and the set $\mathcal{S}_E$ contains all the solutions of $ B_E $, if such solutions exist.
\end{theorem}

\proof
Let $\mathcal{D}$, the set of integer feasible solutions of $MOILFP$ problem, be a finite bounded set contained in $\mathcal{X}$. The cardinality of efficient sets $\mathcal{S}_E$ and $\mathcal {X}_{E}, \mathcal {X}_{E'}$ is a finite number. It contains a finite number of integer solutions. So the search tree would have a finite number of branches. Thus the algorithm terminates in a finite number of steps.\\
For $\mathcal{S}_E$ to contain all the solutions of $(B_E)$, the fathoming rules are used  without loss of any elements in $\mathcal{S}_E$.  At each step $l$ of the algorithm \ref{algo}, if an integer solution $x^{*(l)}$ is found, the cuts  eliminate $x^{*(l)}$ and all dominated solutions from search (see proposition \ref{prop}).

So, the first fathoming  rule is when the set $\mathcal{H}_l$ or $\mathcal{H'}_l$ is empty. In this case, the current node can be pruned since the rest of the domain contains only dominated solutions, either in terms of the multiobjective program or in terms of $ (BIOLFP)$ .
 The second rule is the trivial case when the reduced domain becomes infeasible, whether it is because of previous cuts or the branching.
\endproof

\section{Illustrative example}
Let us consider the optimization of  two utility functions:

\[ f_1(x)=\dfrac{-x_1+x_2-3}{2x_1+x_1+1}  \ \text{ and }f_2(x)=\dfrac{-4x_1+3x_2+1}{2x_1+x_2+2}\ \]
over the efficient set $\X_E$ of the $(MOILFP)$ presented in \citep{KS81}:

\[\left\{
\begin{array}{lll}
\max  \dfrac{x_1-4}{-x_2+2}\\ \noalign{\smallskip}
\max \dfrac{-x_1+4}{x_2+1} \\ \noalign{\smallskip}
\max -x_1+x_2\\ \noalign{\smallskip}
s.t.\\
-x_1+4x_2  \leq 0\\
2x_1-x_2 \leq 8 \\

x_1\geq 0 ,~ x_2\geq 0 \text{ and integers.}
\end{array}\right.\]
To solve this problem using the described algorithm, the steps below will be followed:

\begin{description}
	\item[Initialization] Firstly we initialise $l=0$, $\mathcal{S}_E=\emptyset$ and $\mathcal{X}_0=\mathcal{X}$.

	the forward iterations consist in solving the problem attached to the node, so for the node 0 we will solve $(LFP_0)$:
	\[(LFP_0)
	\left\{
	\begin{array}{lll}
	\max f_1(x)=\dfrac{-x_1+x_2-3}{2x_1+x_1+1}  \\
	s.t.\\
	x\in\X_0

	\end{array}
	\right.
	\]
	
	Since the method of solution of $(LFP_0)$ is a simplex procedure it gives a final simplex table, from which we will analyse the results according to the described algorithm.

	\item[Node 0] solving $(LFP_0)$ gives the final simplex Table \ref{node0} :

\begin{table}[H] \centering
\caption {Optimal simplex table for node 0.}
\begin{tabular}{lccccr} \toprule
$\B_{0}$  &&$x_{3}$& $x_{4}$& &RHS \\
\midrule
$x_{1}$   &&1/7&4/7& &32/7\\
$x_{2}$   &&2/7& 1/7& &8/7\\
\midrule
$\nu^{1(0)}$&&-1/7  & 3/7  & &-45/7\\
  $\mu^{1(0)}$&&-4/7 & -9/7  & &79/7\\
  $\boldsymbol{\gamma^{1(0)}}$&&-37/7&  -24/7&  &-45/79\\
 \midrule
  $\nu^{2(0)}$&&-2/7  & 13/7  &&-97/7 \\
  $\mu^{2(0)}$&&-4/7  & -9/7  &&86/7\\
  $\boldsymbol{\gamma^{2(0)}}$&&-80/7 &  5 &  &-97/86\\
  \midrule
$\eta^{1(0)}$&&-1/7 &-4/7 &&4/7\\ \noalign{\smallskip}
$\vartheta^{1(0)}$&&2/7 &1/7 &&6/7\\ \noalign{\smallskip}
$\boldsymbol{\lambda^{1(0)}}$&&-2/7& -4/7& &2/3\\ \midrule

$\eta^{2(0)}$&&1/7 &4/7 &&-4/7\\ \noalign{\smallskip}
$\vartheta^{2(0)}$&&-2/7 &-1/7 &&15/7\\ \noalign{\smallskip}
$\boldsymbol{\lambda^{2(0)}}$&&1/7& 8/7& &-4/15\\ \midrule

$\eta^{3(0)}$&&-1/7 &3/7 &&-24/7\\ \noalign{\smallskip}
$\vartheta^{3(0)}$&&0 &0 &&1\\ \noalign{\smallskip}
$\boldsymbol{\lambda^{3(0)}}$&&-1/7& 3/7& &-24/7\\
\bottomrule

\end{tabular}
\label{node0}	
\end{table}

	Since the optimal solution $(\frac{32}{7},\frac{8}{7})$ is not  integer, the problem is divided into two sub-problems by adding constraints $ x_1\leq \left\lfloor \frac{32}{7}\right\rfloor $ and $x_1\geq \left\lceil \frac{32}{7}\right\rceil $ respectively to $\X_0$, and we obtain the two programs :
	
	\[\begin{array}{ll}
	(LFP_1)\left\{\begin{array}{ll}
\max f_1(x)=\dfrac{-x_1+x_2-3}{2x_1+x_1+1}  \\
s.t.\\
x\in \mathcal{X}_0\\
x_1\leq 4\\
\end{array}\right.
&(LFP_2)\left\{\begin{array}{ll}
\max f_1(x)=\dfrac{-x_1+x_2-3}{2x_1+x_1+1} \\
s.t.\\
x\in \mathcal{X}_0\\
x_1\geq 5\\
\end{array}\right.
	\end{array}
	\]

\item[Node 1] Solving the problem $(LFP_1) $ gives  the following Table \ref{node1}:	

\begin{table}[H] \centering
\caption {Optimal simplex table for node 1.}
\begin{tabular}{lccccr} \toprule
$\B_{1}$  &&$x_{3}$& $x_{5}$& &RHS \\
\midrule
$x_{1}$   &&0& 1& &4\\
$x_{2}$   &&1/4& 1/4& &1\\
$x_{4}$   &&1/4& -7/4& &1\\
\noalign{\smallskip}  \hline \noalign{\smallskip}
$\nu^{1(1)}$&&-1/4 &3/4 &&-6\\ \noalign{\smallskip}
$\mu^{1(1)}$&&-1/4 &-9/4 &&10\\ \noalign{\smallskip}
$\boldsymbol{\gamma^{1(1)}}$&&-4&-6& &-3/5\\
\midrule
$\nu^{2(1)}$&&-3/4 &13/4 &&-12\\ \noalign{\smallskip}
$\mu^{2(1)}$&&-1/4 &-9/4 &&11\\ \noalign{\smallskip}
$\boldsymbol{\gamma^{2(1)}}$&&-45/4&35/4& &-12/11\\ \midrule

$\eta^{1(1)}$&&0 &-1&&0\\ \noalign{\smallskip}
$\vartheta^{1(1)}$&&1/4 &1/4 &&1\\ \noalign{\smallskip}
$\boldsymbol{\lambda^{1(1)}}$&&0& -1& &0\\
\midrule
  $\eta^{2(1)}$&& 0  & 1  & &0\\
  $\vartheta^{2(1)}$&&-1/4 & -1/4  & &2\\
  $\boldsymbol{\lambda^{2(1)}}$&&0 &  2&  &0\\
\midrule
  $\eta^{3(1)}$&&-1/4 & 3/4  & &-3 \\
  $\vartheta^{3(1)}$&& 0   & 0  & &1\\
  $\boldsymbol{\lambda^{3(1)}}$&&-1/4 & 3/4 &  &-3\\
\bottomrule
\end{tabular}
	\label{node1}
\end{table}

	The optimal solution found $(4,1)$ is integer, according to the algorithm we start by updating $\S$, since this set is empty $\S=\{(4,1)\}$  then we calculate $\H_1$ and $\H_1'$: From the Table \ref{node1} we find $\H_1=\{5\}$ and $\H_1'=\{5\}$. So the Node 0 has one successor Node 3 with $\X_3=\X_1\cap\{x_5\geq 1\}$.

	\item[Node 2] The program $(LFP_2)$ is infeasible.\\

	\item[Node 3] After  solving the  corresponding  linear fractional program, it gives Table \ref{node3}:
\begin{table}[H] \centering
\caption {Optimal simplex table for node 3.}
\begin{tabular}{lccccr} 	\toprule
$\B_{3}$  &&$x_{3}$& $x_{6}$& &RHS \\
\midrule
$x_{1}$   &&0& 1& &3\\
$x_{2}$   &&1/4& 1/4& &3/4\\
$x_{4}  $ &&1/4& -7/4& &11/4\\
$x_{5}  $ &&0& -1& &1\\
\midrule
  $\nu^{1(3)}$&&-1/4  & 3/4  & &-21/4\\
  $\mu^{1(3)}$&&-1/4  & -9/4  & &31/4\\
  $\boldsymbol{\gamma^{1(3)}}$&&-13/4 &  -6 &  &-21/31\\
 \midrule
 $\nu^{^2(3)}$ &&-3/4   & 13/4    &&-35/4 \\
  $\mu^{2(3)}$&&-1/4  & -9/4  && 35/4\\
  $\boldsymbol{\gamma^{2(3)}}$&&-35/4 &35/4  &&-1\\
\midrule
$\eta^{1(3)}$&& 0 & -1 &&-1\\
$\vartheta^{1(3)}$&&1/4 &1/4 &&5/4\\
$\boldsymbol{\lambda^{1(3)}}$&&1/4 & -1 & &-4/5\\ \midrule

$\eta^{2(3)}$&& 0  &1  &&1\\
$\vartheta^{2(3)}$&&-1/4 &-1/4 &&7/4\\
$\boldsymbol{\lambda^{2(3)}}$&&1/4 & 2 & &4/7\\ \midrule

$\eta^{3(3)}$&&-1/4 &3/4 &&-9/4\\
$\vartheta^{3(3)}$&& 0  & 0  &&1\\
$\boldsymbol{\lambda^{3(3)}}$&&-1/4 & 3/4 & &-9/4\\

\bottomrule
\end{tabular}	\label{node3}
\end{table}

	The optimal solution $(3,\frac{3}{4})$ is  not integer so the constraints  $x_2\leq \left\lfloor \frac{3}{4}\right\rfloor$and $x_1\geq \left\lceil \frac{3}{4}\right\rceil$  are added  respectively to  Table \ref{node3} and the nodes 4 and 5 are created.

	\item[Node 4] Solving the problem gives Table \ref{node4}:

\begin{table}[H] \centering
\caption {Optimal simplex table for node 4.}
\begin{tabular}{lccccr} 	\toprule
$\B_{4}$  &&$x_{6}$& $x_{7}$& &RHS \\
\midrule
$x_{1}$   &&0& 1 & &3\\
$x_{2}$   &&1& 0& &0\\
$x_{3}  $ &&-4& -1 & &3\\
$x_{4}  $ &&1 & -2 & &2\\
$x_{5}  $ && 0 & -1& &1\\
\midrule
  $\nu^{1(3)}$&& -1  & 1 & &-6 \\
  $\mu^{1(3)}$&&-1  & -2  & & 7 \\
  $\boldsymbol{\gamma^{1(3)}}$&&-13&  -5 &  &-6/7\\
\midrule
 $\nu^{^2(3)}$ &&-3   & 4    &&-11 \\
  $\mu^{2(3)}$&&-1  & -2  && 8 \\
  $\boldsymbol{\gamma^{2(3)}}$&&-35 &10  &&-11/8\\
\midrule
$\eta^{1(3)}$&&0 &-1 &&-1\\
$\vartheta^{1(3)}$&& 1 & 0 &&2 \\
$\boldsymbol{\lambda^{1(3)}}$&& 1 &  -2 & &-1/2\\ \midrule

$\eta^{2(3)}$&& 0  & 1  &&1\\
$\vartheta^{2(3)}$&&-1 & 0  &&1\\
$\boldsymbol{\lambda^{2(3)}}$&& 1 &  1 & &1\\ \midrule

$\eta^{3(3)}$&&-1 &1  &&-3\\ \noalign{\smallskip}
$\vartheta^{3(3)}$&& 0 & 0 &&1\\ \noalign{\smallskip}
$\boldsymbol{\lambda^{3(3)}}$&& -1 &  1 & &-3\\
\bottomrule
\end{tabular}
\label{node4}
\end{table}	
The optimal solution found $(3,0)$ is integer, this solution is dominated by the solution $(4,1)$ then  $\S=\{(4,1)\}$. From the Table \ref{node4} we find $\H_4=\{6,7\}$ and $\H_4'=\{6\}$, so the Node 4 has one successor Node 6  with $\X_6=\X_4\cap\{x_6+x_7\geq 1,x_6\geq1\}$.  	
	
	\item[Node 5] The corresponding problem is infeasible and the node is fathomed.

	\item[Node 6] By adding the constraints $x_6\geq 1$ and $x_6+x_7\geq1 $  to Table \ref{node4} and solving, we obtain Table \ref{node6}.
	
\begin{table}[H] \centering
\caption {Optimal simplex table for node 6.}
\begin{tabular}{lccccr} 	\toprule
$\B_{6}$  &&$x_{7}$& $x_{9}$& &RHS \\
\midrule
$x_{1}$   && 0 & 1& &2\\
$x_{2}$   &&1 & 0& &0\\
$x_{3}  $ &&-4& 1& &2\\
$x_{4}  $ &&1 & -2& &4\\
$x_{5}  $ && 0& -1& &2\\
$x_{6}  $ &&-0& -1& &1\\
$x_{8}  $ &&-1& -1 & &0\\
\midrule
  $\nu^{1(6)}$&&-1  & 1  & &-5\\
  $\mu^{1(6)}$&&-1  & -2 & &5\\
  $\boldsymbol{\gamma^{1(6)}}$&&-10 &  -5 &  &-1\\
\midrule
  $\nu^{2(6)}$&& -3  & 4  & &-7 \\
  $\mu^{2(6)}$&&-1  &  -2   & &6\\
  $\boldsymbol{\gamma^{2(6)}}$&&-25 &  10 &  &-7/6\\
\midrule

$\eta^{1(6)}$&&0  & -1  &&-2\\
$\vartheta^{1(6)}$&&1  & 0 &&2\\
$\boldsymbol{\lambda^{1(6)}}$&&2 & -2 & &-1\\
\midrule

$\eta^{2(6)}$&& 0  &1 &&2\\
$\vartheta^{2(6)}$&&-1 & 0 &&1\\
$\boldsymbol{\lambda^{2(6)}}$&&2 & 1 & &2\\ \midrule

$\eta^{3(6)}$&&-1 &1 &&-2\\
$\vartheta^{3(6)}$&& 0  & 0 &&1\\
$\boldsymbol{\lambda^{3(6)}}$&&-1 & 1& &-2\\
\bottomrule
\end{tabular}	
\label{node6}
\end{table}
	The optimal solution found $(2,0)$ is integer, this solution is  dominated by the solution $(4,1)$ then  $\S=\{(4,1)\}$. From the Table \ref{node6} we find $\H_6=\{7,9\}$ and $\H_6'=\{9\}$, so the Node 6 has one successor Node 7  with $\X_7=\X_6\cap\{x_7+x_9\geq 1,x_9\geq1\}$.

\item[Node 7] Solving the corresponding problem  produces Table \ref{node7}:

\begin{table}[H] \centering
\caption {Optimal simplex table for node 7.}
\begin{tabular}{lccccr} \toprule
$\B_{7}$  &&$x_{7}$& $x_{10}$& &RHS \\
\midrule
$x_{1}$   &&0& 1& &1\\
$x_{2}$   &&1& 0& &0\\
$x_{3}$   &&-4& 1& &1\\
$x_{4}  $ &&1& -2& &6\\
$x_{5}  $ &&0& -1& &3\\
$x_{6}  $ &&0& -18& &2\\
$x_{8}  $ &&0& -1& &1\\
$x_{9}  $ &&-1& -1& &1\\
$x_{11}  $ &&-1& -1& &0\\
\midrule
$\nu^{1(7)}$&&-1 &1 &&-4\\
$\mu^{1(7)}$&&-1 &-2 &&3\\
$\boldsymbol{\gamma^{1(7)}}$&&-7& -5& &-4/3\\ \midrule

$\nu^{2(7)}$&&-3 &4 &&-3\\
$\mu^{2(7)}$&&1 &-2 &&4 \\
$\boldsymbol{\gamma^{2(7)}}$&&-15&10& &-3/4\\ \midrule

  $\eta^{1(7)}$&& 0  & -1  & &-3\\
  $\vartheta^{1(7)}$&&1  & 0  & &2\\
  $\boldsymbol{\lambda^{1(7)}}$&&3 &  -2 &  &-3/2\\ \midrule
  $\eta^{2(7)}$&&0  & 1  & &3 \\
  $\vartheta^{2(7)}$&&-1  & 0  & &1\\
  $\boldsymbol{\lambda^{2(7)}}$&&3 &  1 &  &3\\
 \midrule
	$\eta^{3(7)}$&&-1 &1 &&-1\\
	$\vartheta^{3(7)}$&&0 &0 &&1\\
	$\boldsymbol{\lambda^{3(7)}}$&&-1& 1& &-1\\
\bottomrule

\end{tabular}
\label{node7}	
\end{table}
	The optimal solution found $(1,0)$ is integer, this solution is  not dominated by the solution $(4,1)$ so  $\S=\{(4,1),(1,0)\}$. From the Table \ref{node7} we find $\H_7=\{7,10\}$ and $\H_7'=\{10\}$, so the Node 7 has one successor Node 8  with $\X_8=\X_7\cap\{x_7+x_{10}\geq 1,x_{10}\geq1\}$.

\item[Node 8] Solving the corresponding problem  produces Table \ref{node8}:
\begin{table}[H] \centering
\caption {Optimal simplex table for node 8}
\begin{tabular}{lccccr} 	\toprule
$\B_{8}$  &&$x_{7}$& $x_{12}$& &RHS \\
\midrule
$x_{1}$   &&0& 1& &0\\
$x_{2}$   &&1& 0& &0\\
$x_{3}  $ &&-4& 1& &0\\
$x_{4}  $ &&1& -2& &8\\
$x_{5}  $ &&0& -1& &4\\
$x_{6}  $ &&0& -1& &3\\
$x_{8}  $ &&0& -1& &2\\
$x_{9}  $ &&-1& -1& &2\\
$x_{10}  $ &&0& -1& &1\\
$x_{11}  $ &&-1& -1& &1\\
$x_{12}  $ &&-1& -1& &0\\
\midrule
$\nu^{1(8)}$&&-1&1 &&-3\\
$\mu^{1(8)}$&&-1 &-2 &&1\\
$\boldsymbol{\gamma^{1(8)}}$&&-4&-5& &-3\\ \midrule

$\nu^{2(8)}$&&-3 &4 &&1\\
$\mu^{2(8)}$&&-1 &-2 &&2\\
$\boldsymbol{\gamma^{2(8)}}$&&-5& 10& &1/2\\

\midrule
  $\eta^{1(8)}$&& 0  & -1  & &-4\\
  $\vartheta^{1(8)}$&&1  & 0  & &2\\
  $\boldsymbol{\lambda^{1(8)}}$&&-4 &  -2 &  &-2\\
\midrule
  $\eta^{2(8)}$&&0  & 1  & &4 \\
  $\vartheta^{2(8)}$&&-1  & 0  & &1\\
  $\boldsymbol{\lambda^{2(8)}}$&&4 &  1 &  &4\\
 \midrule
  $\nu^{3(8)}$&&-1 &1 &&0\\
$\mu^{3(8)}$&&0 &0 &&1\\
$\boldsymbol{\lambda^{3(8)}}$&&-1& 1& &0\\
\bottomrule
\end{tabular}	
\label{node8}
\end{table}

\noindent The optimal solution found $(0,0)$ is integer, this solution is  not dominated by the either   $(4,1)$ or $(0,0)$, it doesn't dominate both   so  $\S=\{(4,1),(1,0),(0,0)\}$. From the Table \ref{node8} we find $\H_8=\{7,13\}$ and $\H_8'=\{13\} $, so the Node 8 has one successor Node 9   with $\X_9=\X_8\cap\{x_7+x_{10}\geq 1,x_{13}\geq1\}$

	\item[Node 9] The corresponding problem is infeasible and the node is fathomed.

\end{description}

\noindent There remains no nodes so the search is finished with the solution set $\S=\{(4,1),(1,0),(0,0)\}$. The efficient set of this problem is $\X_E=\{(4,1),(3,0),(2,0),(1,0),(0,0)\}$

 The search tree is presented in Figure \ref{tree}.

 \begin{figure}[H]
 \tiny
\begin{center}
\begin{tikzpicture}[xscale=0.8,yscale=0.6]
\tikzstyle{fleche}=[->,>=latex,thick]
\tikzstyle{noeud}=[fill=white,circle,draw]
\tikzstyle{feuille}=[fill=lightgray,circle,draw]
\tikzstyle{etiquette}=[midway,fill=white]
\tikzstyle{commentaire}=[left]
\tikzstyle{scommentaire}=[below,black]
\tikzstyle{com}=[right]
\tikzstyle{comi}=[right,blue]
\def\DistanceInterNiveaux{1.5}
\def\DistanceInterFeuilles{4}
\def\NiveauA{(-0)*\DistanceInterNiveaux}
\def\NiveauB{(-1)*\DistanceInterNiveaux}
\def\NiveauC{(-2.7)*\DistanceInterNiveaux}
\def\NiveauD{(-4)*\DistanceInterNiveaux}
\def\NiveauE{(-5.700000000000001)*\DistanceInterNiveaux}
\def\NiveauF{(-7.5)*\DistanceInterNiveaux}
\def\NiveauG{(-9.5)*\DistanceInterNiveaux}
\def\NiveauH{(-11.5)*\DistanceInterNiveaux}
\def\InterFeuilles{(1)*\DistanceInterFeuilles}

\node[noeud] (R) at ({(2.3)*\InterFeuilles},{\NiveauA}) {$0 $};
\draw (R) node[com] at ({(1.7)*\InterFeuilles},{\NiveauA}) {$(\frac{32}{7},\frac{8}{7})$};
\draw (R) node[commentaire] at ({(2.9)*\InterFeuilles},{(1.2)*\NiveauA}) {$\S=\emptyset$};

\node[noeud] (Ra) at ({(1.5)*\InterFeuilles},{\NiveauB}) {$1$};
\draw (Ra) node[comi] at ({(1)*\InterFeuilles},{\NiveauB}) {$(4,1)$};
\draw (Ra) node[commentaire] at ({(2.7)*\InterFeuilles},{(1)*\NiveauB}) {$\small \S=\S\cup\{(4,1)\}$};

\node[noeud] (Raa) at ({(1.5)*\InterFeuilles},{\NiveauC}) {$3$};
\draw (Ra) node[com] at ({(1))*\InterFeuilles},{\NiveauC}) {$(3,\frac{3}{4})$};

\node[noeud] (Raaa) at ({(0.5)*\InterFeuilles},{\NiveauD}) {$4$};
\draw (Raaa) node[com] at ({(0)*\InterFeuilles},{\NiveauD}) {$(3,0)$};

\node[noeud] (Raaaa) at ({(0.5)*\InterFeuilles},{\NiveauE}) {$6$};
\draw (Raaaa) node[com] at ({(0)*\InterFeuilles},{\NiveauE}) {$(2,0)$};

\node[noeud] (Raaaaa) at ({(0.5)*\InterFeuilles},{\NiveauF}) {$7$};
\draw (Raaaaa) node[comi] at ({(0)*\InterFeuilles},{\NiveauF}) {$(1,0)$};
\draw (Raaaaa) node[commentaire] at ({(1.7)*\InterFeuilles},{(1)*\NiveauF}) {$\small \S=\S\cup\{(1,0)\}$};

\node[noeud] (Raaaaaa) at ({(0.5)*\InterFeuilles},{\NiveauG}) {$8$};
\draw (Raaaaa) node[comi] at ({(0)*\InterFeuilles},{\NiveauG}) {$(0,0)$};
\draw (Raaaaa) node[commentaire] at ({(1.7)*\InterFeuilles},{(1)*\NiveauG}) {$\small \S=\S\cup\{(0,0)\}$};

\node[feuille] (Raaaaaaa) at ({(0.5)*\InterFeuilles},{\NiveauH}) {$9$};

\draw (Raaaaaaa.south) node[scommentaire] {$~Infeasible$};

\node[feuille] (Raab) at ({(2.5)*\InterFeuilles},{\NiveauD}) {$5$};
\draw (Raab.south) node[scommentaire] {$Infeasible$};

\node[feuille] (Rb) at ({(3)*\InterFeuilles},{\NiveauB}) {$2$};
\draw (Rb.south) node[scommentaire] {$Infeasible$};
\draw[fleche] (R)--(Ra) node[etiquette] {$x_1\leq4$};

\draw[fleche] (Ra)--(Raa) node[etiquette] {$x_5\geq1$};

\draw[fleche] (Raa)--(Raaa) node[etiquette] {$x_2\leq 0$};
\draw[fleche] (Raaa)--(Raaaa) node[etiquette] {$x_6+x_7\geq1,~x_6\geq1$};
\draw[fleche] (Raaaa)--(Raaaaa) node[etiquette] {$x_{9}+x_7\geq1,~x_{9}\geq1$};
\draw[fleche] (Raaaaa)--(Raaaaaa) node[etiquette] {$x_{10}+x_7\geq1,~x_{10}\geq1$};

\draw[fleche] (Raaaaaa)--(Raaaaaaa) node[etiquette] {$x_{13}+x_7\geq1,~x_{13}\geq1$};

\draw[fleche] (Raa)--(Raab) node[etiquette] {$x_2\geq1$};

\draw[fleche] (R)--(Rb) node[etiquette] {$x_1\geq5$};
\end{tikzpicture}
\end{center}
 \caption{ Search tree of the example}
	 \label{tree}
 \end{figure}
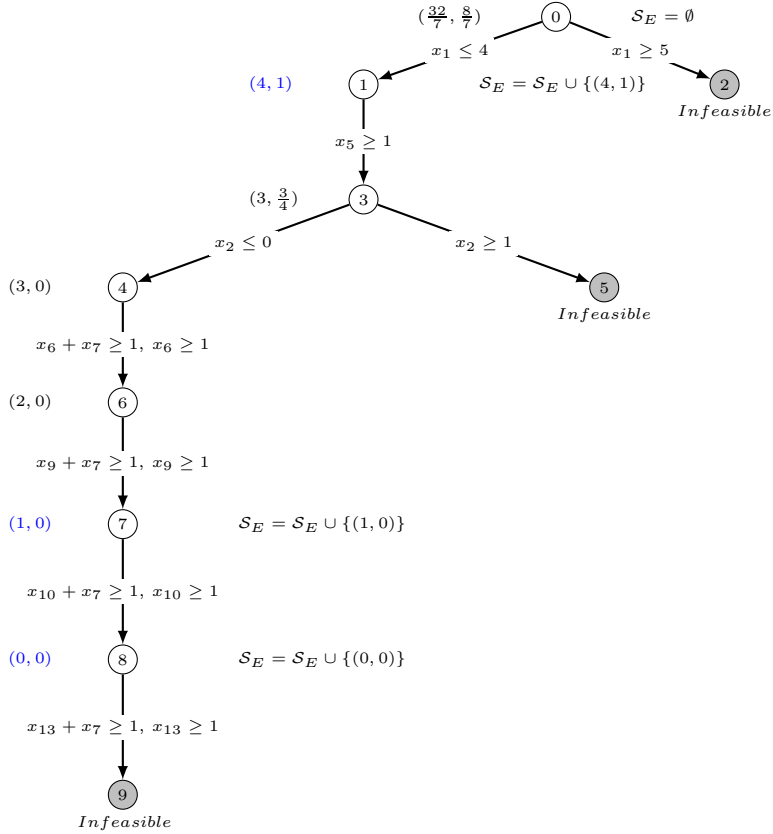
\section{Computational results}
The proposed branch-and-cut algorithm is implemented in  Matlab r2017a. The linear and integer linear programs  are solved  using  the library IBM CPLEX 12.9 for Matlab.
To perform tests we have used a computer with an Intel i7 5500u processor and $8GB$ of memory.

The method is tested on randomly generated $(MOILFP)$ with two randomly generated utility linear fractional functions. The objective functions and constraints coefficients are uncorrelated  uniformly distributed. Each  component of the vector  $b$ and the entries of the matrices A were randomly drawn from discrete uniform distributions in the ranges  [50, 100], [1, 30] respectively. The vectors $p^1,~ p^2$ and $c^i,~i=\overline{1,k}$ and the scalars $\alpha^1,\alpha^2$ and $c^i_0$ are randomly drawn from discrete uniform distributions in the range [-10, 10], while the vectors $q^1,~ q^2$ and $d^i,~i=\overline{1,k}$ and the scalars $\beta^1,\beta^2$ and $d^i_0$ are generated randomly from discrete uniform distributions in the range  [0, 10] to avoid the indefinite case of the null denominator. Also, to avoid infeasibility, all the constraints of each problem are of the $\leq$  type. Furthermore, since all coefficients of A are positive, a bounded feasible region is assured. \\

The problems were grouped according to the number of variables, constraints and objective functions into six categories. In each category the number of objective functions $r$=3,5,7. For each category of problems, 10 instances were solved. \\

The computational results obtained are summarized in Table \ref{tab}. The statistics of the $CPU$ time (in seconds) and the number of nodes are reported. The  last column $\mu$  refers to the average of $\left|\mathcal{X}_E\right|$, while $n$ and $m$ refer to the number of variables and constraints respectively. The cases where the execution time required to find the efficient set $\mathcal{X}_E$ is too high are noted $-$.


\begin{table}[H] \centering
\centering
\caption {Random instances execution time}
\begin{tabular}{ll|lll|lll|l}\toprule

 \multicolumn{2}{l}{ }& \multicolumn{3}{l}{\text{CPU time (second)}} &  \multicolumn{3}{l}{\text{Number of nodes}} & \\  \cline{3-8}

$r\times$ &$m\times n$  &Mean & Max. &Min.  &Mean & Max. &Min. &$\boldsymbol{\mu}$ \\

\midrule
	
3	&10$\times$ 5 &3.20&7.70&0.18& 250 & 600 &10 & 28 \\
	&10$\times$ 10&1.64 &9.20 &0.1 &117.8 & 660 & 3 & 22.3  \\
	&25$\times$ 20 & 13.51 & 28.66& 0.45 & 1257.1 & 2656 & 40 & 26.5\\
	&30$\times$25 & 12.29 & 25.35 & 3.83  & 1167.7 & 2242 & 336 & 26 \\
	&35$\times$30 & 20.25 & 47.89 & 0.44 & 2182.7 & 5892 & 43 & 33.7 \\
	&40$\times$35 & 14.29 & 46.48 & 0.67 &  1862.8 & 5804 & 71 & 37.6  \\ \vspace{0.15cm}
	&45$\times$40 & 22.74 & 54.15 & 1.33  & 3193.2  & 7325  & 81 & 29.2 \\
	&50$\times$ 25&46.25&109.31&10.57&6061.5&13081&1405& -\\
&60$\times$ 30&61.87&222.27&11.70&8012.6&28371&1675& - \\
&70$\times$ 35&162.81&373.31&3.35&20018.5&42385&376 & - \\
&80$\times$ 40&163.27&434.77&0.83&21606.8&52783&35& -\\
&90$\times$ 45&164.48&343.94&4.40&21545.1&42653&592& -\\
&100$\times$ 50&208.17&462.57&45.30&29060.2&60184&5880& - \\
&120$\times$ 60&451.76&1127.12&14.98&54032.4&128250&2106& -\\
&140$\times$ 70&608.84&1015.43&159.95&68329.22&115096&19820& -\\ \vspace{0.15cm}
&160$\times$ 80&517.30&1015.43&6.92&56799.22&115096&596& - \\

5	&5$\times$ 5    &0.57  & 1.81 &0.06 &55.4   &191& 3 & 39.3  \\
    &10$\times$ 5   & 2.15 & 6.23 &0.11 &262.44 & 725& 8 & 149 \\
    &10$\times$ 10  &1.72  & 3.96 &0.31 &218.9  &496 & 38 & 42.5  \\
    &20$\times$ 20  &4.43  & 12.08&0.48 & 571.7 & 1463 & 65 & 109.4 \\
    &30$\times$ 30  &11.01 & 38.81&2.23 & 1477  & 5250 & 281 & 138 \\
    &40$\times$ 40  & 18.84 & 46.37&0.22 & 2138  & 4987 & 16 & 125 \\ \vspace{0.15cm}
    &50$\times$ 50  &  41.88 & 103.93 & 1.79 & 4677.5 & 10157 & 236 & 162.5 \\
 	
7	&5$\times$ 5    &0.26  & 0.56 &0.06 &19.5   &32 &4 & 36.7  \\
    &10$\times$ 5   &2.45  & 6.98 &0.06& 309.1  &824&3 & 221.7 \\
    &10$\times$ 10  &1.40  &2.74 & 0.15 & 199 & 381 &  18 & 110.5  \\
    &20$\times$ 20  &6.00 & 18.57& 0.52 & 958 & 3104 & 60 & 190.9  \\
    &30$\times$ 30  &6.71 & 12.86 & 1.34 & 953.4 & 1733 & 200 &174.3 \\
    &40$\times$ 40  & 30.34 & 54.32 & 17.82 & 4387.5 & 8030 & 2635 & 191.4 \\ \vspace{0.15cm}
    &50$\times$ 50  &  33.4 & 71.02 & 2.34 & 4591 & 9401 & 293 & 243.6 \\
   	 &50$\times$ 25&77.31&163.62&14.05&10553.5&20943&2146& -\\
&60$\times$ 30&55.75&164.07&2.008&8314.1&24890&351 & -\\
&70$\times$ 35&93.86&284.15&0.404&12448.8&32799&40 & - \\
&80$\times$ 40&176.06&350.88&40.98&22950.1&48044&6468& - \\
&90$\times$ 45&243.70&773.29&4.33&28888.8&88885&585 & -\\
&100$\times$ 50&269.34&541.84&51.32&26847.9&52410&5782 & -\\
&110$\times$ 55&517.56&847.28&163.40&51484.9&83932&15147 & -\\
&120$\times$ 60&352.08&899.07&9.58&33535.2&79358&894 & -\\
&130$\times$ 65&546.83&1522.45&19.69&50156&133845&1873 & - \\
&140$\times$ 70&799.93&2243.55&45.90&65041.6&175900&5006 & -\\

	\bottomrule

\end{tabular}
\label{tab}
\end{table}

From the computational experiments shown in Table \ref{tab}, we observe that the proposed method solves small and medium size problems in a reasonable amount of time.

\begin{figure}[H]	
\centering
\begin{tikzpicture}
\tiny
\begin{groupplot} [
  group style={group size=3 by 1},
   height = 4.5cm,
     width = 4.5cm,
  legend pos= north west,
  xmax=11,
  ymax=550,
  xtick={1,...,10},
   xticklabel style={rotate=90,anchor=east}]

]

\nextgroupplot[title= 3 objectives,xticklabels={5x5,10x5,10x10,20x10,25x20,30x25,35x30,40x35,45x40}]

\addplot
coordinates {
   (1,0.60895)
    (2,3.2037)
    (3,1.6374)
    (4,6.3881)
    (5,13.5116)
    (6,12.2906)
    (7,20.2525)
    (8,14.2972)
    (9,22.7486)

    };
 \addplot
    coordinates {
   (1,0.73808)
    (2,6.8457)
    (3,3.2471)
    (4,243.2381)
    (5,33.8309)
    (6,79.5084)
    (7,116.4335)
    (8,153.3144)
    (9,472.8505)
    };

  \legend{Our Method , Brute Force}

\nextgroupplot[title= 5 objectives , xticklabels={5x5,10x5,10x10,20x20,25x20,30x30,35x30,40x40,45x40}]

\addplot
coordinates {
    ( 1,  0.5783)
   (2   , 3.5450)
  ( 3   ,1.7200)
  ( 4  ,  4.4341)
  ( 5 , 10.3784)
   (6  , 11.0101)
  ( 7  , 15.7027)
  ( 8 , 18.8444)
  ( 9  , 39.0803)
   (10  , 41.8870)
    };
 \addplot
    coordinates {
   (1  ,  0.4811)
   (2 ,   5.9393)
   (3 ,   1.3746)
   (4  , 17.9894)
   (5  , 47.8215)
   (6  , 98.0976)
   (7 , 272.3155)
   (8, 737.8233)
   (9 , 643.3598)

    };

  \legend{Our Method , Brute Force}

 \nextgroupplot[title = 7 objectives,xmax=7,xticklabels={5x5,10x5,10x10,20x20,30x30,40x40}]

\addplot
coordinates {
   (1,0.26135)
    (2,2.4521)
    (3,1.4045)
    (4,6.0029)
    (5,6.713)
    (6,30.4359)
    };
 \addplot
    coordinates {
  (1,0.27053)
    (2,5.8503)
    (3,2.6275)
    (4,118.5281)
    (5,323.7137)
    (6,390.001)
    };

  \legend{Our Method ,Brute Force}

\end{groupplot}

\end{tikzpicture}
	\label{comfig}
\caption{Execution time comparison}
\end{figure}
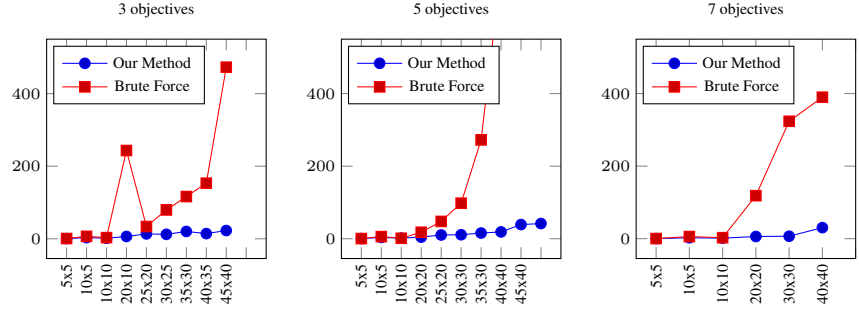

In (Figure \ref{comfig}) summarises the average execution time using our method and the Brute Force method, which consists in finding the two efficient sets $\X_E$ and $ \X_E' $  using the algorithm presented in \citep{CM08}, and then finding the intersection between the two.
Notice that in case where $S_E$ is empty the problem has been replaced by another one.
We observe that our proposal is faster than Brute Force with 3, 5 and 7 objectives, especially when the number of variables and constraints increases.
In fact, when the number of variables and constraints is smaller (5x5, 10x5, 10x10) our method and Brute Force compete with each other, yielding almost the same execution time.

\section{Conclusion}
In this article, we have presented a novel method to optimize a $BIOLFP$ over the efficient set of a $MOILFP$ problem using a branch-and-cut algorithm. The latter reduces the search domain and thus allows us to avoid exploring all efficient sets $ \X_E $ and $ \X_E' $. In addition, the same algorithm can be easily extended in case where there are several utility functions to consider which is a very common real-life case.
The experimental results showed that our proposal yields good execution time for small and medium size problems. Furthermore, when compared with Brute Force method, our proposal proved to be better especially when problem size increases.

For future work, we intend to replace linear fractional  functions with other non-linear functions.

\end{document}